\documentclass[11pt]{amsart}
\usepackage{graphicx}

\usepackage{amssymb}
\usepackage{epstopdf}
\usepackage{amsmath}
\usepackage{amscd}
\usepackage{amsthm,amsfonts,amsxtra,amssymb,amscd}
\usepackage[all]{xy}
\usepackage{enumerate}
\usepackage{hyperref}

\theoremstyle{plain}
\newtheorem*{lemma*}{Lemma}
\newtheorem{lemma}[subsection]{Lemma}
\newtheorem*{theorem*}{Theorem}
\newtheorem{theorem}[subsection]{Theorem}
\newtheorem*{proposition*}{Proposition}
\newtheorem{proposition}[subsection]{Proposition}
\newtheorem*{corollary*}{Corollary}
\newtheorem{corollary}[subsection]{Corollary}
\theoremstyle{definition}
\newtheorem*{definition*}{Definition}
\newtheorem{definition}[subsection]{Definition}
\newtheorem*{example*}{Example}
\newtheorem{example}[subsection]{Example}
\theoremstyle{remark}
\newtheorem*{remark*}{Remark}
\newtheorem{remark}[subsection]{Remark}

 \def\infdex{\text {infdex}}

\def\infdex {\text {{\rm infdex}}}
\newcommand{\g}{\mathfrak{g}}

\title{Infinitesimal index: cohomology computations }
\author{
C. De Concini,\quad
C. Procesi,\quad M. Vergne}
\dedicatory{Dedicated to Tonny Springer}
\begin{document}
\begin{abstract}
 In this note we study the equivariant cohomology with compact supports
 of the zeroes of the moment map for  the cotangent bundle of a linear
 representation of a torus  and some of its notable subsets, using the
 theory of the infinitesimal index, developed in \cite{dpv34}. We show
 that, in analogy  to the case of equivariant $K$-theory   dealt with in
 \cite{dpv1} using  the index of transversally elliptic operators, we
 obtain isomorphisms with notable spaces of splines  studied in
 (\cite{DM1}, \cite{DM3}). \end{abstract}
\maketitle
\section{Introduction} Let $G$ be a compact Lie group with Lie algebra $\mathfrak g$, $N$ a manifold with $G$-action and equipped with a $G$--equivariant 1--form $\sigma$.

From this setting, one has a moment map  $\mu^\sigma:N\to\mathfrak g^*$.  A particularly important case is that of $N=T^*M$, the cotangent bundle of a manifold with a $G$ action, equipped with the canonical action form. In this case, the zeroes of the moment map is a subspace $T^*_GM$  whose equivariant $K$--theory is  strongly related to the index of transversally elliptic operators as shown in \cite{At}.

In order to understand explicit formulas for such an index, in \cite{dpv34} we have introduced the {\em infinitesimal index} infdex, a map from the equivariant cohomology with compact supports of the zeroes of the moment map to distributions  on $\mathfrak g^*$.

We have proved several properties for this map which, at least in the case of  the space  $T^*_GM$, in principle allow us to reduce the computations to the case in which $G$ is a torus   and the manifold is  a complex linear representation of  $G$. A finite dimensional complex representation of a torus  is   the direct sum  of one dimensional representations given by characters. If $X$ is a list of characters, we denote by $M_X$ the corresponding linear representation which is naturally filtered by open sets $M_{X,\geq i}$ where the dimension of the orbit  is $\geq i$.

In this paper, we first compute the equivariant cohomology of the open sets  $M_{X,\geq i}$, and also of some slightly more general open sets in $M_X$. This part of our paper, namely Sections 2 and 3, does not use the notion of infinitesimal index. The results are obtained from the structure of the algebra $S[\mathfrak g^*][(\prod_{a\in X}a)^{-1}]$ as a module over the Weyl algebra studied in \cite{dp1}.

In Section 4 we apply the results we have  obtained to the equivariant cohomology of the open set $M_X^{fin}$  of points with finite stabilizer.  Using Poincar\'e duality, we remark that the equivariant cohomology with compact supports $H^*_{G,c}(T_G^*M_X^{fin})$ of $T_G^*M_X^{fin}$ is isomorphic as a $S[\mathfrak g^*]$-module to a remarkable  finite dimensional  space $D(X)$ of polynomial functions on $\mathfrak g^*$, where $S[\mathfrak g^*]$ acts by differentiation. The space $D(X)$ is  defined as   the space of                                                             solutions of a set of linear partial differential equations combinatorially associated to $X$ and has been of importance   in approximation theory (see for example \cite{DM1}, \cite{DM3}).

At this point the notion of infinitesimal index comes into play. We show in Theorem \ref{finitu} that the infinitesimal index gives an isomorphism between $H^*_{G,c}(T_G^*M_X^{fin})$ and $D(X)$.  After this, we show that, for each $i$,    the infinitesimal index  establishes an isomorphism between $H_{G,c}^*(T_G^*M_{X,\geq i})$  and a space of {\em splines} $\tilde{\mathcal G}_{i }(X) $,   introduced in \cite{dpv},  (cf. \eqref{itilg})  and generalizing $D(X)$.

It should be mentioned that, in the previous paper  \cite{dpv1}, similar results have been proved, using the index of transversally elliptic differential operators, in order to compute the equivariant K-theory of the spaces $T_G^*M_{X,\geq i}$.

This paper represents a sort of  ``infinitesimal" version of \cite{dpv1} and  will  be used, in a forthcoming  paper \cite{dpv6},   to  give explicit formulas for  the index of transversally elliptic operators.

%\begin{acknowledgments}
Finally it is a pleasure to thank Michel Brion for a number of useful conversations and remarks and the referee for a very careful reading and extremely useful suggestions.

%\end{acknowledgments}

            \section{A special module}
 \subsection{A module filtration}

 Let $G$ be a compact torus with Lie algebra $\mathfrak g$ and character group $\Lambda$. We are going to consider $\Lambda$ as a lattice in $\mathfrak g^*$.

 We need to recall some general results  proved in \cite{dp1}.  Let us fix a list $X=(a_1,\ldots ,a_m)$   of non zero characters in $\Lambda\subset \mathfrak g^*$.
 Let $S[\mathfrak g^*]$ be the symmetric algebra on $\mathfrak g^*$ or in other words the algebra of polynomial functions  on $\mathfrak g$. For a list $Y$ of  vectors, let us set $d_Y:=\prod_{a\in Y }a\in S[\mathfrak g^*]$.

     \begin{definition}
A subspace  $\underline s$ of $\mathfrak g^*$ is called {\it rational} (relative to $X$) if  $\underline s =\langle X\cap \underline s \rangle $.
\end{definition}
In general if $A$ is a set of vectors  we denote by $\langle A \rangle $ the linear span of $A$.

We shall denote by $\mathcal S_X$ the set of rational subspaces and, for a given $0\leq k\leq s$, by $\mathcal S_X(k)$ the set of the rational subspaces of dimension $k$.

We shall assume always that $X$ spans $\mathfrak g^*$ and need to recall that a {\it cocircuit}
in $X$  is a sublist of $X$    of
the form $Y:= X\setminus H$  where $H$ is a rational hyperplane.
 % The  polynomials  $d_Y:=\prod_{a\in Y} a\in S[\mathfrak g^*]$, as $Y$ runs over the cocircuits, give a system of polynomial equations  $d_Y=0$.
 \begin{definition}\label{IIX}
We denote by  $I_X$ the   ideal in $S[\mathfrak g^*]$ generated by the elements $d_Y$'s, as $Y$ runs over the cocircuits.
\end{definition}

One knows that $I_X$  defines a scheme $V_X$  supported at 0 and of length $d(X)=\dim(S[\mathfrak g^*]/I_X)$ where  $d(X)$ equals the number of bases extracted from $X$,  (see \cite{dp1}, Theorem 11.13).

 Consider the localized algebra
   $R_X:=S[\mathfrak g^*][d_X^{-1}]$, which is the coordinate ring of the complement of the hyperplane arrangement   in  $\mathfrak g$, defined by the equations $a=0,\ a\in X$.

This algebra is a cyclic module  over the Weyl algebra $W[\mathfrak g]$ of differential operators  with polynomial coefficients, generated by $d_X^{-1}$ (Theorem 8.22 of  \cite{dp1}).

In \cite{dp1}, we have seen that  this $W[\mathfrak g]$-module has a canonical filtration,  the {\it filtration by polar order}, where we put in degree of filtration $\leq k$  all fractions in which the denominator is a product of elements in $X$ spanning a  rational subspace of dimension $\leq k$ (we say that   $k$    is the polar order on the boundary divisors).  We denote this subspace by $R_{X,k}$.  One of the important facts  (Theorem 8.10 in \cite{dp1}) is that
\begin{theorem}\label{ciso}
The  $W[\mathfrak g]$-module $R_{X,k}/ R_{X,k-1}$ is semisimple, its isotypic components are in 1--1 correspondence with the rational subspaces  of  dimension $k$ and such a isotypic component is generated by the class of $1/d_{ X\cap \underline s}$.
\end{theorem}\smallskip

Consider the rank 1 free  $S[\mathfrak g^*]$ submodule $L:= d_X^{-1}S[\mathfrak g^*]$  in $R_X$ generated by $d_X^{-1}$. Set   $L_k:=L\cap R_{X,k}$, that is the intersection of $L$ with the $k-$filtration.
We obtain     for each $k$ an ideal $I_k$ of $S[\mathfrak g^*]$  defined by $$I_k:=L_kd_X.$$

 For a given rational subspace  $\underline s$  of dimension $k$, denote by  $I_{\underline s}:=S[\mathfrak g^*] d_{ X\setminus \underline s}$ the principal ideal generated by $d_{ X\setminus \underline s}$. Notice that
$$I_{\underline s}L= d_{ X\cap \underline s}^{-1}S[\mathfrak g^*]\subset L_k .  $$
Thus $I_{\underline s}\subset I_k$
 and indeed from Theorem 11.29 of   \cite{dp1} one gets  $$I_k=\sum_{\underline s\in\mathcal S_X(k)}I_{\underline s}.$$ If  $Q\subset \mathcal S_X$ is a set of rational subspaces, we set $$I_{Q}=\sum_{\underline s\in Q}I_{\underline s}$$  for the ideal generated by the elements $d_{ X\setminus \underline s}$ for $\underline s\in Q$.

Associated to $\underline s$,  we also consider  the list $ X\cap \underline s$   consisting of those elements of $X$ lying in $\underline s$ and we may consider the ideal $I_{X\cap {\underline s}}\subset S[\underline s]$, as defined in \ref{IIX}, and its extension $J_{X\cap {\underline s}}:=I_{X\cap {\underline s}}S[\mathfrak g^*]$. Since $S[\mathfrak g^*]$ is a free $S[\underline s]$ module, the obvious map
$$S[\mathfrak g^*]\otimes_{S[\underline s]}I_{X\cap {\underline s}}\to J_{X\cap {\underline s}}$$
is an isomorphism and \begin{equation}\label{firstiso}S[\mathfrak g^*]/J_{X\cap {\underline s}}\simeq S[\mathfrak g^*]\otimes_{S[\underline s]}(S[\underline s]/I_{X\cap {\underline s}}).\end{equation}

%$$\left(\begin{array}{cc}34 & 23 \\11 & 0\end{array}\right)$$

 \begin{lemma}
If $\underline s$  is of  dimension $k$, we have that  $d_{X\cap {\underline s}}^{-1}S[\mathfrak g^*]\subset L_k$ and   \begin{equation}
\label{unin}d_{X\cap {\underline s}}^{-1}S[\mathfrak g^*]\cap L_{k-1}\supset d_{X\cap {\underline s}}^{-1}J_{X\cap {\underline s}}.
\end{equation}
\end{lemma}\begin{proof}
We have already remarked the first statement. As for the second, by definition $J_{X\cap {\underline s}}$ is the ideal generated by the elements  $d_Z$  where $Z$ is a cocircuit in $X\cap {\underline s}$.   This means   that $ Z$ is contained in $X\cap {\underline s}$  and that $Y:=(X\cap {\underline s})\setminus  Z$ spans a subspace of dimension $k-1$.  Hence $d_{X\cap {\underline s}}^{-1}d_ZS[\mathfrak g^*]=d_Y^{-1}S[\mathfrak g^*]\subset L_{k-1}.$
\end{proof} Multiplying  Formula \eqref{unin}  by $d_X$, we deduce that
 \begin{equation}\label{ladina} I_{   \underline s} \cap I_{k-1}\supset J_{X\cap {  \underline s}}d_{ X\setminus \underline s} =\sum_{\underline t\subset \underline s, \ \underline t\in \mathcal S_X(k-1)}I_{ \underline t}.\end{equation}

   In this way, multiplication by $d_{X\cap {\underline s}}^{-1}$ gives   an homomorphism of $S[\mathfrak g^*]$--modules      $j_{\underline s}: S[\mathfrak g^*]/J_{X\cap {\underline s}}\to L_k/L_{k-1}$ and hence, taking direct sums, a homomorphism $ j:=\oplus_{\underline s\in \mathcal S_X(k)}j_{\underline s}$
   \begin{equation}\label{loiso}j : \oplus_{\underline s\in \mathcal S_X(k)}S[\mathfrak g^*]/J_{X\cap {\underline s}}\to L_k/L_{k-1}.
   \end{equation}

We have (Theorem 11.3.15 of  \cite{dp1}):
\begin{theorem}\label{filtr}  The homomorphism $j$ is an isomorphism.
\end{theorem}
Using (\ref{ladina}),  Theorem \ref{filtr} tells us that the summation morphism
 \begin{equation}\label{loiso2}\tilde j:\oplus_{\underline s\in \mathcal S_X(k)}I_{\underline s}/J_{X_{  \underline s}}d_{ X\setminus \underline s}\to I_k/I_{k-1}  \end{equation}
is an isomorphism.

\begin{definition} A set $Q\subset \mathcal S_X$ of rational subspaces is called admissible if, for every $\underline s\in Q$,  $Q$ also contains all rational subspaces $\underline t\subset\underline s$.
\end{definition}

 From Theorem \ref{filtr},  we deduce
\begin{proposition}\label{lamain} 1) For any subset $\mathcal G\subset \mathcal S_X(k)$
\begin{equation}(\sum_{\underline s\in\mathcal G}I_{\underline s})\cap I_{k-1}=\sum_{\underline t\subset \underline s\in\mathcal G, \ \underline t\in \mathcal S_X(k-1)}I_{ \underline t}.\end{equation}

2) Given an admissible subset $Q\subset \mathcal S_X $ and a rational subspace  $\underline s\in  Q$ of maximal  dimension $k$,  then
\begin{equation}\label{laint}
I_{  \underline s} \cap I_{Q\setminus\{\underline s\}} =I_{  \underline s}\cap I_{k-1}=\sum_{\underline t\subset \underline s, \ \underline t\in \mathcal S_X(k-1)}I_{ \underline t}.
\end{equation}
\end{proposition}
\begin{proof} 1)  By (\ref{loiso2}), the restriction of $\tilde j$ to $ \oplus_{\underline s\in \mathcal G}I_{\underline s}/J_{X_{  \underline s}}d_{ X\setminus \underline s}$ is injective. It follows that
$$(\sum_{\underline s\in \mathcal G}I_{\underline s})\cap I_{k-1}=\sum_{\underline s\in \mathcal G}J_{X_{  \underline s}}d_{ X\setminus \underline s}=\sum_{\underline t\subset \underline s\in\mathcal G, \ \underline t\in \mathcal S_X(k-1)}I_{ \underline t}$$
as desired.

2) We first assume that $Q\supset \mathcal S_X(k-1)$ so that   $Q\setminus \underline s=\mathcal S_X(k-1)\cup \mathcal G$ with $\mathcal G\subset \mathcal S_X(k)$. If $\mathcal G$  is empty,  then $I_{Q_{\setminus \{ \underline s\}}}=I_{k-1}$ and our claim is a special case of 1).

Otherwise
$I_{Q_{\setminus \underline s}}=I_{k-1}+ \sum_{\underline t\in\mathcal G}I_{\underline t}.$
Let $b\in  I_{  \underline s} \cap I_{Q_{\setminus \underline s}}$. Passing modulo $I_{k-1}$,  we get an element lying in $I_{\underline s}/(I_{\underline s}\cap I_{k-1})$ and in $(\sum_{\underline t\in\mathcal G}
I_{\underline t})/((\sum_{\underline t\in\mathcal G}I_{\underline t})
\cap I_{k-1})$.
 But the restriction of $\tilde j$ to $\oplus_{\underline t\in \mathcal \mathcal G\cup\{\underline s\}}I_{\underline t}/I_{\underline t}\cap I_{k-1}$ is injective. It follows that $b\in I_{k-1}$ as desired.

Passing to the general case, set $\tilde Q=Q\cup\mathcal S_X(k-1)$. We have
$$I_{  \underline s} \cap I_{Q\setminus\{\underline s\}} \subset I_{  \underline s} \cap I_{\tilde Q\setminus\{\underline s\}} =I_{  \underline s}\cap I_{k-1}=\sum_{\underline t\subset \underline s, \ \underline t\in \mathcal S_X(k-1)}I_{ \underline t}.$$
On the other hand it is clear that
$$I_{  \underline s} \cap I_{Q\setminus\{\underline s\}} \supset \sum_{\underline t\subset \underline s, \ \underline t\in \mathcal S_X(k-1)}I_{ \underline t}$$
and our claim follows.
\end{proof}

\section{Equivariant cohomology}
\subsection{Equivariant cohomology of $M_{X,\geq k}$}

Let $G$ be a compact torus.
Given a $G$ space $M$, we denote   for simplicity by   $H^*_G(M)$ the $G$ equivariant cohomology $H^*_G(M,\mathbb R)$ of $M$  with real coefficients.

For a character $a\in\Lambda$, we denote by $L_a$ the one dimensional complex $G$ module on which $G$ acts via  $a$. Given a list $X$ in $\Lambda$, we set $$M_X=\oplus_{a\in X}L_a.$$

Our purpose is to compute the equivariant cohomology of various $G$ stable open sets in $M_X$.

To begin with, since $M_X $ is a vector space, $H^*_G(M_X  )$ equals the equivariant cohomology of a point and thus  $H^*_G(M_X )=S[\mathfrak g^*]$, and    $\mathfrak g^*=H_{G }^2(M_X )$.

Let $X$  and $M_X$ be as before and $Y$ a sublist of $X$.  We have $M_Y \subset M_X$.
\begin{lemma}\label{facil}
$H^*_G(M_X\setminus M_Y)=S[\mathfrak g^*]/(d_{X\setminus Y}).$
\end{lemma}
\begin{proof}  Since $M_X\setminus M_Y=(M_{X\setminus Y}\setminus 0)\times M_Y$, we have $H^*_G(M_X\setminus M_Y)\cong  H^*_G(M_{X\setminus Y}\setminus 0).$ Moreover, the long exact sequence  of the pair  $(M_{X\setminus Y}, 0 )$ and the definition of the equivariant Euler class yield  $ H^*_G(M_{X\setminus Y}\setminus 0)\cong H^*_G(M_{X\setminus Y})/(d_{X\setminus Y})$.

%
%Write $M_X=M_{X\setminus Y}\oplus M_Y$ and denote by $\pi:M_X\to M_Y$ the projection onto the second factor. This is a $G$ equivariant vector bundle on $M_Y$ with fiber $M_{X\setminus Y}$. Thus its equivariant Euler class in $H_{G }^*(M_Y)=S[\mathfrak g^*]$ is given by $d_{X\setminus Y}$. The space $M_X\setminus M_Y$ is obtained by  removing the zero section of $\pi$. It is   a standard fact that
%$H^*_G(M_X\setminus M_Y)$ equals the equivariant cohomology of $M_Y$ modulo the ideal generated by the Euler class, that is $S[\mathfrak g^*]/(d_{X\setminus Y}).$
\end{proof}
Take a subset $Q\subset \mathcal S_X$ of rational subspaces and set
$$\mathcal A_Q=M_X\setminus \cup_{\underline s\in Q}M_{X\cap {\underline s}}.$$
\begin{theorem}\label{ilpr}
   $H^*_G(\mathcal A_Q)$ is isomorphic as a graded ring to
$S[\mathfrak g^*]/I_Q$.

 In particular $\mathcal A_Q$ has no $G$ equivariant odd cohomology.
\end{theorem} \begin{proof}  Let us add to $Q$ all the rational subspaces $\underline t$ which are contained in at least one of the elements of $Q$. In this way, we get a new subset $\overline Q\supset Q$ which is now admissible and is such that $\mathcal A_Q=\mathcal A_{\overline Q}$.
Also it is clear that  $I_Q=I_{\overline Q}$.

Having made this remark, we may without loss of generality assume that $Q$ is admissible. If $Q=\emptyset$, then $\mathcal A_\emptyset =M_X$, the ideal $I_\emptyset=\{0\}$ and there is nothing to prove. Thus we can proceed by induction on the cardinality of $Q$ and assume that $Q$ is nonempty.

Take $\{\underline s\}$ maximal in $Q$. Notice that $Q\setminus \{\underline s\}$ is also admissible. Furthermore  the set
  $$\mathcal S_{< \underline s}=\{\underline t\in\mathcal S_X\ |\ \underline t\subsetneq \underline s\}$$
  is also  admissible and strictly contained in $Q$.

We have
$$ \mathcal A_Q = \mathcal A_{Q\setminus\{\underline s\}}\cap (M_X\setminus  M_{X\cap {\underline s} })$$ and
$$\mathcal A_{Q\setminus\{\underline s\}}\cup (M_X\setminus  M_{X\cap {\underline s}} )=\mathcal A_{\mathcal S_{< \underline s}}.$$
Thus, by induction, we have \begin{equation}\label{isos}H^*_G( \mathcal A_{Q\setminus\{\underline s\}})=S[\mathfrak g^*]/I_{Q\setminus\{\underline s\}},\ \ \ H^*_G( \mathcal A_{Q\setminus\{\underline s\}}\cup (M_X\setminus  M_{X\cap {\underline s}} )=S[\mathfrak g^*]/I_{\mathcal S_{< \underline s}}.\end{equation}
Set $Y:= X\cap {\underline s}$.  Consider the homomorphism
$$\psi:H^*_G( \mathcal A_{Q\setminus\{\underline s\}}\cup (M_X\setminus  M_{X\cap {\underline s}} ))\to H^*_G( \mathcal A_{Q\setminus\{\underline s\}})\oplus H^*_G(M_X\setminus M_Y)$$
induced by inclusion.
Using the isomorphisms (\ref{isos}) and Lemma \ref{facil}, we get a commutative diagram
$$\begin{CD}H^*_G( \mathcal A_{Q\setminus\{\underline s\}}\cup (M_X\setminus  M_{X\cap {\underline s}} ))@>\psi >> H^*_G( \mathcal A_{Q\setminus\{\underline s\}})\oplus H^*_G(M_X\setminus M_Y)\\ @V\simeq VV @V\simeq VV\\ S[\mathfrak g^*]/I_{\mathcal S_{< \underline s}}@>  >>S[\mathfrak g^*]/I_{Q\setminus\{\underline s\}}\oplus S[\mathfrak g^*]/(d_{X\setminus Y}) \end{CD}$$
where the vertical arrows are isomorphisms.
Now by Proposition \ref{lamain}  2)
$$I_{  \underline s} \cap I_{Q\setminus\{\underline s\}} =I_{  \underline s}\cap I_{k-1}=\sum_{\underline t\subset \underline s, \ \underline t\in \mathcal S_X(k-1)}I_{ \underline t}=I_{\mathcal S_{< \underline s}}.$$
Thus $\psi$ is injective. We immediately deduce from the Mayer-Vietoris sequence that the homomorphism
$$\phi: H^*_G( \mathcal A_{Q\setminus\{\underline s\}})\oplus H^*_G(M_X\setminus M_Y)\to H^*_G( \mathcal A_{Q}) $$
is surjective and that $H^*_G( \mathcal A_{Q})\simeq S[\mathfrak g^*]/I_Q$ as desired.
\end{proof}
\begin{remark}
There is a parallel theorem
%, somewhat easier to prove,
for the algebraic  counterpart of equivariant cohomology, that is the equivariant Chow ring (see Edidin and Graham \cite{EG}).
\end{remark}

\subsection{Equivariant cohomology of $M_{X,\geq k}$}

Let $s:=\dim(G)$ and let us look at some special cases of Theorem \ref{ilpr}.  We always assume that $X$ spans  the  $s$-dimensional space $\mathfrak g^*$, which is equivalent to assume that the generic point of $M_X$ has a finite stabilizer.

 If $Q=\mathcal S_X(k-1)$, $$\mathcal A_{\mathcal S_X(k-1)}=M_X\setminus \cup_{\underline s\in \mathcal S_X(k-1)}M_{X\cap {\underline s}} :=M_{X,\geq k}$$ is the set of points whose orbits have dimension at least $k$.  \begin{definition}
For $k=s$,  $M_{X,\geq s}$ is  the open set of points with finite stabilizer that we also denote by $M_X^{fin}$.
\end{definition}

\begin{corollary}\label{finit}
The equivariant cohomology of  $M_{X,\geq k}$ is isomorphic as a graded algebra to $S[\mathfrak g^*]$ modulo the ideal $I_{k-1}$. In particular  $H^*_{G}(M_X^{fin})=S[\mathfrak g^*]/I_X$ with $I_X$ the ideal generated by the elements $d_Y$ as $Y$ runs over the cocircuits.

\end{corollary}

\begin{remark}
Assume that $X$ spans $\Lambda\subset \mathfrak g^*$ and that the cone   $C(X)\subset \mathfrak g^*$ of linear combinations of the vectors in $X$ with non negative coefficients is acute. Consider the complexified torus $G_{\mathbb C}$. The list $X$ gives an embedding $G\to (S_1)^m$ and its complexification $G_{\mathbb C}\to (\mathbb C^*)^m$, so we may consider the  torus $T=(S_1)^m/G$ and its complexification $T_{\mathbb C}= (\mathbb C^*)^m/G_{\mathbb C}$.

Write $z\in M_X$ as $z=\sum_a z_a$ with $z_a\in L_a$.
Choose $\xi\in C(X)$  not lying in any rational hyperplane (in this case we say that $\xi$ is regular).
Then the set $P_{\xi}:=\{z\in M_X\,|\, \sum_a |z_a|^2 a=\xi\}$ is smooth,  contained in $M_X^{fin}$, and $P_\xi/G$ is a toric variety for the complex torus $T_{\mathbb C}$.  As a $T_{\mathbb C}$-variety,  $P_\xi/G$ depends only on the connected component of the set of regular points containing $\xi$. Furthermore  $P_\xi/G$ is projective and rationally smooth.

Generators and relations for the ring $H_G^*(P_\xi)=H^*(P_\xi/G)$
are well known (see for example \cite{danilov}). In particular,
$S(\g^*)$ surjects on $H^*(P_\xi/G)$.
Consider the restriction map
$H_G^*(M_X^{fin})\to H^*_G(P_\xi)$.
Thus this map is surjective for any regular $\xi$ and  its kernel  (which depends upon the
connected component of the set of regular points containing $\xi$)
is  generated by the polynomials $d_{X\setminus \sigma}\in S(\g^*)$,
where $\sigma\subset X$ runs  over the bases of $\mathfrak g^*$
such that $\xi$ is not in the
cone generated by  $\sigma$.

\end{remark}

\begin{remark}
It may be interesting to observe that to $X$, as to any matroid, is associated a two variable polynomial, the Tutte polynomial,  that describes the statistics of  external and internal activity.  Then the statistic of external activity gives rise to the Betti numbers of equivariant cohomology  of  $M_{X}^{fin}$  while from internal activity one deduces the characteristic polynomial that describes Betti numbers of the  complement of the complex hyperplane arrangement deduced from $X$.  A direct topological interpretation of the Tutte polynomial has been recently obtained in \cite{FS}.
\end{remark}

\section {Equivariant cohomology of $T_G^*M$}

\subsection{The space $D(X)$\label{vet}}  In order to perform our cohomology computations, we need first to introduce some new spaces.
We keep the notation of the previous sections.

Given $a\in \mathfrak g^*$, let us denote by $\partial_a$  the  derivative in the $a$ direction. We identify $S[\mathfrak g^*]$ to the space of differential operators  with constant coefficients on $\mathfrak g^*$.

To  a cocircuit $Y$, we associate the differential operator $\partial_{Y}:=\prod_{a\in Y}\partial_a$.
\begin{definition}\label{ladedix}
The space $D(X)$  is given by
\begin{equation}\label{ilsist} D(X):=\{f\in S[\mathfrak g]\,|\, \partial_Yf=0,\ \  {\rm for\ every\ cocircuit\ }Y\}.\end{equation}
\end{definition}
The space $D(X)$ is stable by the action of  $S[\mathfrak g^*]$.

Notice that, by its definition:
\begin{remark}\label{ladua}
$D(X)$ is the (graded) vector space dual to the algebra $D^*(X)=S[\mathfrak g^*]/I_X$, that is the cohomology ring $H^*_G(M_{X}^{fin})$ by Corollary \ref{finit}.
\end{remark}  To be consistent with grading in cohomology, we double the degrees in $S[\mathfrak g]$ and hence in $D(X)$ and we set for each $i\geq 0$, $D(X)^{2i+1}=\{0\}$.

Using the Lebesgue measure associated to the lattice $\Lambda$, we will in what follows freely identify polynomial functions on $\g^*$ with polynomial densities on $\g^*$.

The polynomials in $D(X)$, dual to the algebra $D^*(X):=S[\mathfrak g^*]/I_X$,  can be naturally interpreted as Laplace--Fourier transforms  of  the finite dimensional space  $\hat D(X)$  of those generalized  functions   which  vanish on the functions vanishing on the scheme $V_X$.  \smallskip

Denote by $\mathcal S'(\mathfrak g^*)$ the space of tempered distributions on $\mathfrak g^*$.
Assume now that there is an element $x\in \mathfrak g$ such that $\langle x,a\rangle>0$ for every $a$ in $X$. Recall that the {\it multivariate spline $T_X$}   is the tempered distribution
defined by:
\begin{equation}\label{multiva}
\langle T_X\,|\,f\rangle = \int_0^\infty\dots\int_0^\infty
f(\sum_{i=1}^mt_i a_i)dt_1\dots dt_m.
\end{equation}
Its Laplace transform  is  $d_X^{-1}:=1/\prod_{a\in X}a $. Notice that, if $a\in X$, \begin{equation}\label{lecon}T_X=T_a* T_{X\setminus a},\
\partial_aT_X=T_{X\setminus a}, \implies \partial_XT_X=T_{\emptyset}=\delta_0 .
\end{equation}

Let $\underline r$ be a vector subspace in $\mathfrak g^*$. We have an embedding $j:\mathcal S'(\underline r)\to \mathcal S'(\mathfrak g^*)$ by $j(\phi)(f)=\phi(f|\underline r)$ for any $\phi\in \mathcal S'(\underline r)$, $f$ a Schwartz function on $\mathfrak g^*$. We denote the image $j(\mathcal S'(\underline r))$ by $\mathcal S'(\mathfrak g^*, \underline r) $ (sometimes we even identify $ \mathcal S'(\underline r) $ with $\mathcal S'(\mathfrak g^*, \underline r) $  if there is no ambiguity).
 We next define
 the vector space:\begin{definition}\label{gra}\begin{equation}
\label{ilpp}  \mathcal G(X):=
\{f\in {\mathcal S'}(\mathfrak g^*)\,|\,
\partial_{X\setminus\underline r}f\in \mathcal S'(\mathfrak g^*, \underline r ),
\text{ for all } \underline r\in \mathcal S_X\}.
\end{equation}
\end{definition}

\begin{example} Let $G=S^1$ and identify $\Lambda$ with $\mathbb Z$ and $\mathfrak g^*$ with $\mathbb R$. Let $X=1^{k+1}=\underbrace{(1,1,\ldots,1)}_{k+1}$.

Then there are two rational subspaces: $\mathbb R$ and the origin. The only cocircuit is $X$ itself and $\partial_X=\frac{d^{k+1}}{dx^{k+1}}$. The space
$D(X)$ consists of the polynomials of degree $\leq k$ and $T_X=x^k/k!$ if $x\geq 0$ and $0$ otherwise.
%So the polynomial $x^k/k!$ equals $T_X+(-1)^kT_{-X}$.
It is easy to see that  $\mathcal G(X)=D(X)\oplus \mathbb R T_X.$
\end{example}

We are now going to recall a few  properties of $\mathcal G(X)$ (see also \cite{dpv1}). For this, given a list of non zero vectors $Z$ in $\mathfrak g^*$, we consider the dual hyperplane arrangement, $a^\perp\subset \mathfrak g$, $a\in Z$. Any connected component $F$ of the complement of this arrangement is called a {\em  regular face} for $Z$. An element $\phi\in F$ decomposes $Z=A\cup B$ where $\phi$ is positive on $A$ and negative on $B$. This decomposition depends only upon $F$. We define
\begin{equation}
\label{tef}T_Z^F=(-1)^{|B|}T_{(A,-B)}.
\end{equation}
Notice that $T_Z^F$ is supported on the cone $C(A,-B)$ of non negative linear combinations of the vectors in the list $(A,-B)$.

Take
 the subset  $\mathcal S_X{(i)}$ of  subspaces  $\underline r\in \mathcal S_X $ of  dimension $i$.
Consider   $\partial_{X\setminus\underline r}$ as an operator on
$\mathcal G(X)$ with values in $\mathcal S'(\mathfrak g^*,\underline r)$.
 Define the spaces
\begin{equation}
\label{CAFI}\mathcal G(X)_{i}:=\cap_{\underline t\in \mathcal S_X{(i-1)}}\ker(
 \partial_{X\setminus \underline t}).
\end{equation}

Notice that   by definition $\mathcal G(X)_0=\mathcal G(X)$, that
$\mathcal G(X)_{\dim \mathfrak g^*}$ is the space  $D(X)$ and that $\mathcal G(X)_{i+1}\subseteq
\mathcal G(X)_{i}$.

\begin{remark}
Consider a polynomial density  $g\in D(X\cap \underline r) $, a face $F_{\underline r}$ defining $X\setminus\underline r=A\cup B$ and  ${T }_{X\setminus\underline r}^{F_{\underline r}} $.  The convolution ${T }_{X\setminus\underline r}^{F_{\underline r}}*g$ is well defined since, for any $z\in \mathfrak g^*$,  the set of pairs  $x\in C(A,-B),y\in  \underline r$ with $x+y=z$ is compact.
\end{remark}

\begin{lemma}\label{lalari}
Let $\underline r \in  \mathcal S_X{(i)}$.

\quad i)  The image of $\partial_{X\setminus\underline r}$ restricted
to $\mathcal G(X)_{i}$ is contained in    $D(X\cap \underline r)$.

\quad ii) Take
rational subspaces $\underline t$  and $\underline r$. For any $g\in D(X\cap \underline r),$
\begin{equation}\label{madre}\partial_{X\setminus \underline t}( {T }_{X\setminus\underline r}^{F_{\underline r}}*g) =
( \partial_{(X\setminus \underline t)\setminus \underline r
}{T }_{X\setminus\underline r}^{F_{\underline r}})* (\partial_{(X\cap \underline r)\setminus
(\underline t \cap \underline r)}g).\end{equation}

\quad iii) If $g$ is in  $D(X\cap \underline r)$, then
${T }_{X\setminus\underline r}^{F_{\underline r}}*g\in \mathcal G(X)_{i}$.
\end{lemma}
\begin{proof}
{\it i)} First we know, by the definition of $\mathcal G(X)$, that
$\partial_{X\setminus\underline r} \mathcal G(X)_{i}$ is contained in the space
$\mathcal S'(\mathfrak g^*, \underline r )$. Let $\underline t$  be a rational
hyperplane of $\underline r$, so that $\underline t$ is of
dimension $i-1$. By definition, we have that for every $f\in
\mathcal G(X)_{i}$
$$0=\prod_{a\in X\setminus \underline t}\partial_a f=
\prod_{a\in (X\cap \underline r)\setminus \underline t} \partial_a
\partial_{X\setminus\underline r}f.$$ This means that
$\partial_{X\setminus\underline r} f$ satisfies the differential
equations given by the cocircuits of   $X\cap \underline r$, that is,  it lies in $D(X\cap \underline r)$.

\smallskip

{\it ii)} We have that
 $\partial_{X\setminus \underline t} =
 \partial_{(X\setminus \underline t)  \cap \underline r }
 \partial_{(X\setminus \underline t)\setminus \underline r }$ but $ \partial_{(X\setminus \underline t)  \cap \underline r }=
 \partial_{(X\cap \underline r)\setminus (\underline t \cap \underline r)} $. Thus
$$\partial_{X\setminus \underline t}( {T }_{X\setminus\underline r}^{F_{\underline r}}*g) =
( \partial_{(X\setminus \underline t)\setminus \underline r
}{T }_{X\setminus\underline r}^{F_{\underline r}})* (\partial_{(X\cap \underline r)\setminus
(\underline t \cap \underline r)}g)$$
as desired.

iii) If $\underline t$ does not contain $\underline r$,  we get that
$(\partial_{(X\cap \underline r)\setminus
(\underline t \cap \underline r)}g)=0$ and hence, by \eqref{madre},
$$\partial_{X\setminus \underline t}( {T }_{X\setminus\underline r}^{F_{\underline r}}*g)=0.$$

\end{proof}

Consider the map $\mu_i: \mathcal G(X)_i \to \oplus_{\underline r\in
 \mathcal S_X(i)} D(X\cap \underline r)$ given by $$\mu_i
f:=\oplus_{\underline r\in  \mathcal S_X(i)} \partial_{X\setminus
\underline r}f$$ and the map ${\bf P}_i:\oplus_{\underline r\in
 \mathcal S_X(i)} D(X\cap \underline r)\to \mathcal G(X)_i$ given by
$${\bf P}_i(\oplus g_{\underline
r}):= \sum T_{X\setminus\underline r}^{F_{\underline r}}*g_{\underline r }.$$

\begin{theorem}\label{lemexact}
The sequence $$0\longrightarrow \mathcal G(X)_{i+1} \longrightarrow \mathcal G(X)_i
\stackrel{\mu_i}\longrightarrow\oplus_{\underline r\in
 \mathcal S_X(i)} D(X\cap \underline r)\longrightarrow 0$$ is
exact. Furthermore, the map ${\bf P}_i$ provides a splitting  of
this exact sequence, i.e. $\mu_i {{\bf P}_i}={\rm Id}$.
\end{theorem}
\begin{proof}
By definition, $\mathcal G(X)_{i+1}$ is the kernel of $\mu_i$,  thus we only need to show that
$\mu_i {{\bf P}_i}={\rm Id}$.
Given   $\underline r\in  \mathcal S_X(i)$ and $g\in D(X\cap \underline r)$,  by Formula \eqref{madre} we have  $\partial_{X\setminus \underline r}({T }_{X\setminus\underline r}^{F_{\underline r}}*g)=g.$  If instead we take
 another subspace  $\underline t\neq \underline r$ of $ \mathcal S_X(i)$, then
$\underline r\cap \underline t$ is a proper subspace of
$\underline t$. As we have seen above, if $g\in D(X\cap \underline r)$, $\partial_{X\setminus
\underline t} ({T }_{X\setminus\underline r}^{F_{\underline r}}*g) =0.$ Thus,  given a family $g_{\underline r}\in D(X\cap \underline r)$, the
function $f= \sum_{\underline t\in  \mathcal S_X(i)} {T
}_{X\setminus\underline t}^{F_{\underline t}}*g_{\underline t}$ is such that
$\partial_{X\setminus \underline r}f=g_ {\underline r}$ for all $\underline r\in   \mathcal S_X(i)$. This proves our claim
that  $\mu_i {{\bf P}_i}={\rm Id}$.\end{proof}

Putting together these   facts, we immediately get

\begin{theorem}\label{gestad1}
 Choose,  for every rational space $\underline r$,
  a  regular face $F_{\underline r}$ for  $X\setminus \underline r$. Then:
\begin{equation}\label{ladeco}\mathcal G(X)=\oplus_{\underline r\in \mathcal S_X}
T_{X\setminus r}^{F_{\underline r}}*D(X\cap \underline r ).\end{equation}
\end{theorem}

\begin{corollary} The dimension of $\mathcal G(X)$ equals the number of sublists of $X$ which are linearly independent.
\end{corollary}
\begin{proof}. This follows immediately from \eqref{ladeco} and the fact (see for example \cite{dp1}  Theorem 11.8) that $D(X)$ has dimension equal to the number of bases which can be extracted from $X.$
\end{proof}

%\begin{proof}

%\end{proof}

We define
\begin{equation}\label{itilg}
\tilde {\mathcal G}(X)=S[\mathfrak g^*]\mathcal G(X),\quad \tilde {\mathcal G}_i(X)=S[\mathfrak g^*]\mathcal G_i(X)
\end{equation} where the elements in $S[\mathfrak g^*]$ act on distributions as differential operators with constant coefficients.
\begin{remark}\label{lapes} If we set $$D^{\mathfrak g}(X\cap
\underline r )=S[\mathfrak g^*] D (X\cap
\underline r )\cong S[\mathfrak g^*]\otimes_{S[(\mathfrak g/\mathfrak g_{\underline r})^*]}D(X\cap \underline r),$$ Theorem \ref{lemexact}, together with the fact that the maps $\mu_i$ and $P_i$ extend to $S[\mathfrak g^*]$-module maps (which we denote by the same letter), immediately implies that we have an exact sequence of $S[\mathfrak g^*]$-modules
$$0\to\tilde{\mathcal G}_{i+1}(X) \to  \tilde{\mathcal G}_{i}(X)  \stackrel{\mu_i}\to
\oplus_{\underline r\in  \mathcal S_X{(i)}}D^{\mathfrak g} (X\cap
\underline r )\to 0.$$\end{remark}
Furthermore one can give generators for  $\tilde {\mathcal G}(X)$ as a $S[\mathfrak g^*]$-module  as follows:

\begin{theorem}\label{gestad2}
$$\tilde {\mathcal G}(X)=\sum_{F}S[\mathfrak g^*]T_X^F$$ as $F$ runs over all regular faces for $X$.\end{theorem}
\begin{proof}
 Denote by $M$ the $S[\mathfrak g^*]$ module generated by the elements $T_X^F$, as $F$ runs on all regular faces.   In general,  from  the description of $\tilde{\mathcal G}(X)$ given in Formula
 \eqref{ladeco},  it is enough to prove that  elements of the  type   $
T_{X\setminus \underline r}^{F_{\underline r}}*g$   with $g\in D (X\cap
\underline r)$ are in $M$. As $D (X\cap \underline r)\subset \mathcal G(X\cap \underline r)$,
it is sufficient to prove by induction that each element
$T_{X\setminus \underline r}^{F_{\underline r}}*T_{X\cap \underline r}^{K}$
is in $M$,  where $K$  is any regular face for the system  $X\cap
\underline r$.
We choose a   linear function $u_0$ in the face $F_{\underline r}$. Thus $u_0$   vanishes on  $\underline r$ and  is non zero on every element $a\in X$ not in $\underline r$. We choose a linear function $u_1$ such that the restriction of $u_1$ to $\underline r$ lies in the face $K$.  In particular, $u_1$ is non zero on every element $a\in X\cap \underline r$.
 We can choose $\epsilon$ sufficiently small such that $u:=u_0+\epsilon u_1$ is  non zero on every element $a\in X$. Then $u_0+\epsilon u_1$ defines a regular face $F$.  We see that a vector $a\in X\setminus \underline r$ is positive for  $u$ if and only if it is positive for $u_0$, similarly  a vector $a\in X\cap \underline r$ is positive for  $u$ if and only if it is positive for $u_1$, hence from \eqref{lecon}, and  the definition \eqref{tef},  it follows that  ${T  }_{X\setminus \underline r}^{F_{\underline r}}*{T}_{X\cap \underline r}^{K}$ is equal to ${T}_X^F$.
\end{proof}\bigskip

This construction has a discrete counterpart, thoroughly studied in \cite{dpv1} and related to the study of the index of transversally elliptic operators and of computations in equivariant K-theory in which differential operators are replaced by difference operators.

\section{Equivariant cohomology with compact supports of $T^*M_X$.}

\subsection{Equivariant cohomology with compact supports and the infinitesimal index.} Let $G$ be a compact Lie group, in \cite{dpv34} we have introduced
 a de Rham model for the equivariant cohomology $H^*_{G,c}(Z)$ with compact supports of a $G$-stable closed subset $Z\subset N$ of a $G$-manifold $N$.
 A representative of an element  in $ H^*_{G,c}(Z)$ is a compactly supported equivariant form on $N$ such that $D\alpha $ is zero on a neighborhood of $Z$. Two representatives $\alpha_1,\alpha_2$ agree if $\alpha_1-\alpha_2=D\beta+\gamma$ where $\beta,\gamma$ are compactly supported and $\gamma$ vanishes on a neighborhood of $Z$.

 Furthermore assume that we have a $G$-equivariant one form $\sigma$ on $N$ called an action form. We define the corresponding moment map
 $\mu:N\to \mathfrak g^*$ by setting for any $u\in \mathfrak g$, $n\in N$, $\mu (n)(x):=-\langle \sigma, v_x\rangle(n)$, $v_x$   being the vector field  on $N$ corresponding to $x$.   Set $\Omega(x):=D\sigma(x)=d\sigma+\mu(x)$  ($D$ is the equivariant differential).

 If we take as $Z$ the zeroes $N^0=\mu^{-1}(0)$ of the moment map, we have then  defined a map of $S[\mathfrak g^*]$-modules
 $${\rm infdex}:H^*_{G,c}(N^0)\to  \mathcal S'(\mathfrak g^*)^G$$
 called infinitesimal index. For $\alpha(x)$ a form giving a cohomology class $[\alpha]\in H^*_{G,c}(N^0)$ and $f$ a test function we have:\begin{equation}\label{definf}
 \langle \infdex_G^{\mu}([\alpha]),f\rangle = \lim_{s\to \infty} \int_N
 \int_{\mathfrak g}
 e^{is \Omega(x)} \alpha(x) \hat f(x) dx
 \end{equation}
This is a well defined map from $H_{G,c}^*(N^0)$ to invariant
distributions on $\mathfrak g^*$. We refer to  \cite{dpv34}  for the proof of most of the properties of $H^*_{G,c}(Z)$ and of the infinitesimal index which we are going to use in what follows.

 We are going to study the case in which we start with a $G$-manifold $M$. We set $N= T^*M$ and we take the canonical one form $\sigma$. In this case it follows immediately from the definitions that $(T^*M)^0$ equals the space $T^*_GM$ whose  fiber over a point $x\in M$ is  formed by all the cotangent vectors $\xi\in T^*_xM$  which vanish on the tangent space to the orbit of $x$  under $G$, in the point $x$. Thus   each fiber $(T^*_GM)_x$ is a linear subspace  of $T_x^* M$. In general the dimension of $(T^*_GM)_x$  is not constant and this space is not a vector bundle on $M$.

\subsection{Connection forms and   the Chern-Weil map}
We shall use a fundamental notion in Cartan's theory of equivariant cohomology. Let us recall
\begin{definition}
Given an action of a compact Lie group $G$ on a manifold $P$ with finite stabilizers, a {\em connection form} is a $G$-invariant one  form  $\omega\in \mathcal A^1(P)\otimes \mathfrak g$  with coefficients in the Lie algebra of $G$ such that $-\iota_x \omega=x$ for all $x\in  \mathfrak g$.

\end{definition}

If on $P$ with free $G$ action we also have a commuting action of another group $L$,  it is easy to see that there exists
     a $L\times
G$ invariant connection form $\omega\in \mathcal A^1(P)\otimes \mathfrak g$  on $P$ for the free action of $G$.

Let $Q:=P/G$.
Define
the curvature $R $  and the $L$-equivariant curvature $R_y $ of the bundle
$P\to Q$ by \begin{equation}
\label{ledc} R:=d\omega+\frac{1}{2}[\omega,\omega],\quad R_y:=-i_y\omega+R.
\end{equation}

We have the Chern-Weil map $c:S[\mathfrak g^*]\to H^*_{G}(P)\cong H^*(P/G)$ defined by $p \mapsto [p(R)]$ (see \cite{dpv34} p.8).
Through this map we give to $H^*(P/G)$  and $H^*_c(P/G)$  a  $S[\mathfrak g^*]$ module structure.
\begin{proposition}\label{podu} If $G$ acts freely (or with finite stabilizers) on a manifold $ P $,
the  Poincar\'e duality for $Q=P/G$ commutes with the   $S[\mathfrak g^*]^G$-module structures.
 \end{proposition}
\begin{proof}   This   depends upon the fact that the $S[\mathfrak g^*]^G$-module structure of $H^*(Q)$ comes from the Chern--Weil morphism $S[\mathfrak g^*]^G\to H^{even}(Q)$ determined by the bundle.   The $S[\mathfrak g^*]^G$-module structure of $H^*_c(Q)$ comes from the same morphism and the fact that $H^*_c(Q)$ is a  $H^*(Q)$  module under multiplication and finally that duality is  given by integration formula $\int_Q\alpha\wedge \beta$  with $\alpha, \beta$ closed and $\beta$ with compact support.
\end{proof}

\subsection{The equivariant cohomology of $T^*_GM_X^{fin}$}
 Our task is now to use the infinitesimal index to compute the equivariant cohomology with compact supports of $T^*_GM_X$ and more generally of $T_G^*M_{X,\geq k}$. Notice that if we consider ordinary equivariant cohomology, it is immediate by
  $G$-homotopy equivalence to deduce

\begin{proposition}\label{eqcoh} The equivariant cohomology  of the space  $T_G^*M_{X,\geq k}$ equals that of  $ M_{X,\geq k}$ for all $k$.
\end{proposition}
  We have already remarked that,
 in the case $k=s$, we have $M_{X,\geq k}=M_{X}^{fin}$ and that $H^*_{G}(M_X^{fin})=D^*(X)$. Now, since $G$ acts on $M_X^{fin}$ with finite stabilizers, and we use cohomology with real coefficients, we get that $H^*_{G}(M_X^{fin})=H^* (M_X^{fin}/G)$ and by Poincar\'e duality
\begin{equation}\label{dual}H^h_{G,c}(M_X^{fin})=H^h_c (M_X^{fin}/G)=(H^{2|X|-s-h} (M_X^{fin}/G))^*=D^{2|X|-s-h}(X).\end{equation}

Now, in order to compute the equivariant cohomology with compact supports of $T_G^*M_X^{fin}$, we need some well known  general considerations.

%We begin by computing the equivariant cohomology with compact support of a
%  real finite dimensional $G$-module $M$.
%  \begin{proposition} As a graded $S[\mathfrak g^*]$-module, $H^*_{G,c}(M)$ is free of rank one with a generator in degree equal to $\dim M$.
%  \end{proposition}
% \begin{proof} Choose a positive definite $G$ invariant scalar product on $  M$. Take the orthogonal sum  $\tilde M$ of $M$ and $\mathbb R$ with the standard Euclidean structure and trivial $G$ action.   Take the unit sphere $S\subset \tilde M$ pointed by the $G$-fixed point $p=(0,1)\in \tilde M$.
%
% We have that $H^*_{G,c}(M)=H^*_G(S,p)$.  Since $\tilde M$ contains a trivial component, its equivariant Euler class equals to $0$. It then follows that there is a Thom class $\tau_M\in H_G^{\dim M}(S,p)=H^{\dim M}_{G,c}(M)$ restricting to the fundamental class of $S$ and that
% $H^*_G(S,p)=S[\mathfrak g^*]\tau_M$.
% \end{proof}
%
%
%The above construction can be easily globalized to get  a Thom isomorphism, for equivariant bundles. So,

  Let $N$ be a $G$-manifold,  $\mathcal M$ be a   $G$-equivariant vector bundle on $N$ of rank $r$ with projection $p:\mathcal M\to  N$. Then (see \cite{mat-qui}), there is an equivariant Thom form $\tau_{\mathcal M}$, which can be taken supported in any arbitrarily small disk bundle   around $N$ in $\mathcal M $, such that  in particular:
\begin{proposition}\label{thom} The map
$$C: H^*_{G,c}(N)\to H^{*+r}_{G,c}(\mathcal M)$$
defined by $C(\alpha)=p^*(\alpha)\wedge \tau_{\mathcal M}$ is  an isomorphism
\end{proposition}   Let $Z$ be an  oriented $G$ manifold  and $s:N\hookrightarrow Z$ a $G$-stable oriented submanifold.
 Assume that $N$ has an action form
  with moment map  $\mu$ and that  $Z$ is equipped with an action form   such that the associated moment map $\mu_Z$ extends $\mu$. Thus $Z^0\cap N=N^0$.

We have used the Thom form in \cite{dpv34} in order to define a map $$s_!:H^*_{G,c}(N^0)\to H^*_{G,c}(Z^0).$$ In particular we can apply this when $Z= N\times M_B$,  $s:N\to N\times M_B$ is the embedding of the 0 section and $M_B$ is the  linear representation associated with some list $B$ of non zero vectors in $\Lambda$ equipped with an action form with the origin lying in $M_B^0$.  Then we  take on $N\times M_B $ the action form given by the sum of the action forms on $N$ and $M_B$ and we get the map
$$s_!:H^*_{G,c}(N^0)\to H^*_{G,c}((N\times M_B )^0).$$
We have
\begin{proposition}\label{lamo}
If $[\lambda]\in  H^*_{G,c}( (N\times M_B)^0)$, then:
$$s_!s^*([\lambda])=d_B[\lambda]. $$
\end{proposition}
 \begin{proof}
 We first  want a Poincar\'e Lemma for a $G$ manifold $N$ and a   vector space $M_B$ with $G$ action.

Consider the map $q:N\times M_B\times[0,1]\to N\times M_B,\ q(x,y,t)=(x,ty)$  and the maps $i_t: N\times M_B\to N\times M_B\times[0,1], i_t(x,y)=(x,y,t).$

Given an equivariant form $\lambda=\lambda(x,y),\ x\in N,y\in M_B$ on $ N\times M_B$,  define the forms $\tilde \lambda,\bar  \lambda$  on            $N\times M_B\times[0,1]$    by
$$q^*\lambda= \tilde \lambda +dt\wedge \bar\lambda $$
where $ \tilde \lambda$  does not contain $dt$ and can be thought of as the form $q_t^*\lambda$, and now $q_t:(x,y)\to (x,ty)$.
Set $$I(\lambda):=\int_0^1 i_t^*\bar\lambda dt.$$
We claim that we have the homotopy formula
\begin{equation}\label{homo}
\lambda(x,y)-\lambda(x,0)=  D I\lambda+ID\lambda .
\end{equation}
In fact
$$DI(\lambda):=\int_0^1 i_t^*D\bar\lambda dt,\   q^*D\lambda=D \tilde \lambda -dt\wedge D\bar\lambda .$$ Write $D \tilde \lambda=\omega +dt\wedge \eta,$ so that $ID\lambda= \int_0^1 i_t^* ( \eta - D\bar\lambda) dt$
gives
$$ DI(\lambda)=-   ID \lambda+\int_0^1 i_t^* \eta dt.$$ If we think of $ \tilde \lambda$ as a form on $N\times M_B$ depending on $t$, we see that $\eta=\frac{d}{dt} \tilde \lambda =\frac{d}{dt} q_t^*\lambda$. It follows that $$  \int_0^1 i_t^* \eta dt= q_1^*\lambda-q_0^*\lambda . $$

We now multiply Equation (\ref{homo}) above by a Thom form $\tau$ for the trivial bundle $p:N\times M_B\to N$ which is the pull back of a Thom form $\tau_0$ on $M_B$ under the projection $N\times M_B\to M_B$.  We may assume the support of $\tau$ as close to $N$ as we wish, that is in $N\times B_\epsilon$  with $ B_\epsilon$ a ball of radius $\epsilon$ centered in the origin. In particular, take a form $\lambda$ such that $D\lambda$ has support  $K$  disjoint   from  $(N\times M_B)^0$.  For such a form, the support of  $ ID\lambda$  is contained in the set of points  $(x,y)$ such that the segment $(x,ty),\ t\in[0,1]$, intersects  $K$.  The support of  $\tau  ID\lambda$  is contained in the previous set of points  $(x,y)$ but with the further requirement that $y\in B_\epsilon$.

Then, if $\epsilon$ is small, we can make sure that the set $\cup_{t\in [0,1]  } t(K\cap (N\times B_\epsilon))$ is still disjoint from  $(N\times M_B)^0$.

We deduce that $$\tau  \lambda(x,y)-\tau \lambda(x,0)=  D\tau I\lambda+\tau ID\lambda ,$$
the forms in this equality have all compact support and moreover  the support of $\tau ID\lambda $  is disjoint from   $(N\times M_B)^0$.

This implies that $[\tau  \lambda(x,y)]=[\tau \lambda(x,0)]$ in the cohomology $H^*_{G,c}((N\times M_B)^0)$.

Now  by definition  $\lambda(x,0)=p^*s^*\lambda,$ so $ [\tau \lambda(x,0)]=s_!s^*[\lambda].$

 The inclusion of the origin   gives an algebra isomorphism between $H^*_G(M_B)$ and  $H^*_G(pt)=S[\mathfrak g^*]$ and the image of  the class of $\tau_0$ is the Euler class of $M_B$ which is $d_B$.    Thus $[\tau  \lambda]=[d_B\lambda]$.
  \end{proof}
 \begin{remark}\label{ciccio} Notice that, if we assume that the moment map on $M_B$ equals $0$, then $(N\times M_B)^0=N^0\times M_B$ and the map $s_!$ is just the Thom isomorphism between $H^*_{G,c}(N^0)$ and $H^*_{G,c}(N^0\times M_B)$.\end{remark}

 Let us now go back to our computations.  The projection $p:T_G^*M_X^{fin}\to M_X^{fin}$ is a real vector bundle of rank $2|X|-s$ so that, applying Proposition \ref{thom}, we get  $H^*_{G,c}(T_G^*M_X^{fin})=H^{*+2|X|-s}_{G,c}(M_X^{fin})$. Thus putting together this with \eqref{dual}, we get
\begin{proposition}\label{shift}
As  a graded $S[\mathfrak g^*]$-module, \begin{equation}
\label{edd}H^*_{G,c}(T_G^*M_X^{fin})\simeq D^{4|X|-2s-*}(X).
\end{equation} In particular $T_G^*M_X^{fin}$ has no equivariant odd cohomology with compact supports.
\end{proposition}
\begin{proof}  We apply Proposition \ref{podu} and deduce that $H^*_{G,c}(T_G^*M_X^{fin})$  is isomorphic to the dual of  $H^*_{G }(T_G^*M_X^{fin})$ as $S[\mathfrak g^*]$-modules, where the dual structure is given by $\langle  a \phi, u\rangle= \langle   \phi, a u\rangle,a \in S[\mathfrak g^*],$ $
\phi\in  H^*_{G }(T_G^*M_X^{fin})^*$, $ u\in  H^*_{G }(T_G^*M_X^{fin})$. The statement now follows from Remark \ref{ladua}.
\end{proof}

%\newpage
\subsection{The equivariant cohomology of $T^*_GM_{X,\geq i}$}
It is now interesting to interpret formula \eqref{edd} via the theory of the infinitesimal index.  To do this, we need to recall a few facts.
As we have already remarked, the action of $G$ on $T_G^*M_X^{fin}$ is essentially free, so, denoting by $Q$ the quotient $T_G^*M_X^{fin}/G$, we can take an equivariant $\mathfrak g$-valued curvature form $R $   for the map $T_G^*M_X^{fin}\to Q$.

We have the Chern-Weil map $c:S[\mathfrak g^*]\to H^*_{G}(T_G^*M_X^{fin})$ defined by $p \mapsto [p(R)]$.

  Take a closed equivariant form with compact support $\gamma$  on $Q$.  We can apply the Theory of the infinitesimal index  to the class $[\gamma]\in H^*_{c}(Q)\simeq  H^*_{G,c}(T_G^*M_X^{fin})$.
By Proposition 4.20 of \cite{dpv34}    the infinitesimal index is given by the polynomial density on $\mathfrak g^*$
\begin{equation}\label{linind}(\int_Q\gamma e^{i\langle R,\xi\rangle}) d\xi.\end{equation} We have
\begin{theorem}\label{finitu} The map ${\rm infdex}$  is a graded isomorphism as $S[\mathfrak g^*]$-modules  of $H^*_{G,c}(T_G^*M_X^{fin})$ onto $D(X)$.
 \end{theorem}
 \begin{proof} Given $p\in S[\mathfrak g^*]$, it defines at the same time a cohomology class $c(p)\in  H^*_{G }(T_G^*M_X^{fin}) $ (given by the Chern--Weil morphism) and also  a differential operator with constant coefficients on $\mathfrak g$, we shall write it as $p(\partial)$.
 Now notice that the  Poincar\'e  duality pairing $([\gamma],c(p))$ is given by
 $$\int_Q\gamma p(R)=\Big (\int_Q\gamma p(R) e^{i\langle R,\xi\rangle}\Big )_{|\xi=0}= (p(\partial) {\rm infdex}(\gamma))_{|\xi=0}.$$
 By Theorem \ref{ilpr} we have  in our case that the Chern-Weil map  $c:S[\mathfrak g^*]\to   H^*_{G }(T_G^*M_X^{fin}) $ is surjective with kernel $I_X$. From this and the previous considerations everything follows.\end{proof}

 % \part{Equivariant differential cohomology}

In order to proceed we need to recall a few definitions.
For a rational subspace $\underline s$,
we may consider the subspace  $M_{\underline s}:=\oplus_{a\in X\cap \underline s}L_a$ of $M_X$,  we also denote by   $G_{\underline s}$   the subgroup of $G$ joint kernel of the   elements in  $X\cap \underline s$. The group $ G_{\underline s}$  acts trivially on $M_{\underline s}$ inducing an action  of $G/G_{\underline s}$.
\begin{definition} We define the set $M_{\underline s}^f$  to be the open set of $M_{\underline s}$ where $G/G_{\underline s}$ acts with  finite stabilizers.
\end{definition}
 Thus $$H^*_{G,c} (M_{\underline s}^f)=S[\mathfrak g^*]\otimes_{S[(\mathfrak g / \mathfrak g_{\underline s} )^*]}H^{*}_{ G/G_{\underline s},c} (M_{\underline s}^f)$$
where $\mathfrak g^*$ is in degree 2.
In particular, by    Proposition \ref{shift}, we deduce that $H_{G,c}^{2i+1}( T^*_G M_{\underline s}^f)=0$.

Now set $\tilde T^*_G M_{\underline s}^f:=T_G^*M_X|_{ M_{\underline s}^f}$,  the restriction of  $T_G^*M_X$ to $M_{\underline s}^f$.
We see that  $\tilde T^*_G M_{\underline s}^f= T^*_G M_{\underline s}^f\times M_{X\setminus\underline s}^*$,
  so we have a Thom  isomorphism
   $$C_{\underline s}:  H_{G,c}^{2i}(  T^*_{ G }  M_{\underline s}^f)\to H_{G,c}^{2(i+|X\setminus\underline s|)}(\tilde T^*_G M_{\underline s}^f),\quad H_{G,c}^{2i+1}(\tilde T^*_G M_{\underline s}^f)=0.$$
Choose $0\leq i\leq s$.
We pass now to study the $G$-invariant open subspace  $M_{X,\geq i}$ of $M$.
The set  $M_{X,\geq i+1}$ is open in $M_{X,\geq i}$ with complement the  set $M_{=i}$,  disjoint union of the sets $M_{\underline s}^f$ with  $\underline s \in \mathcal S_X{(i)}$. Denote by $\tilde T_G^* M_{=i}$ the restriction of $T^*_G M$ to $M_{=i}$,  disjoint union of the sets
$\tilde T^*_G M_{\underline s}^f.$
Denote $$j:  M_{X,\geq i+1}\to M_{X,\geq i}\quad\text{ the open inclusion}$$ $$e: \tilde T_G^* M_{=i}=\cup_{\underline s  } \tilde T^*_G M_{\underline s}^f\to T_G^*M_{X,\geq i}\quad\text{  the closed embedding.}$$ Let $C_i$ be the Thom isomorphism from $H_{G,c}^{2i}( T^*_G M_{=i})$ to $H_{G,c}^{2(i+|X\setminus\underline s|)}(\tilde T^*_G M_{=i})$, the direct sum of the Thom isomorphisms $C_{\underline s}$.

\begin{theorem}\label{Tmenos} For each $0\leq i\leq s-1$,
\begin{enumerate}[a)]\item For each $h\geq 0$, $H_{G,c}^{2h+1}(T^*_GM_{X,\geq i})=0$.
\item
For each $h\geq 1$, the following sequence is exact
\begin{equation}\label{lasecondina}0\to H_{G,c}^{ 2h}(T_G^*M_{X,\geq i+1})\stackrel{j_*}\to H_{G,c}^{ 2h}(T_G^*M_{X,\geq i})\stackrel{C_{i}^{-1}e^*}\to\oplus_{\underline s\in {\mathcal S}_X{(i)}}  H_{G,c}^{ 2h-2|X\setminus\underline s|}(T_G^*M_{\underline s}^f)\to 0 .\end{equation}
\end{enumerate}
\end{theorem}
\begin{proof}
Since $M_{X,\geq s}= M^{fin}_X$, we can assume by induction on $s-i$, that {\it a)} holds for each $j>i$. Also since $M_{=i}$ is the disjoint union of the spaces $M_{\underline s}^f$ which have no odd equivariant cohomology with compact supports,
we get that $H_{G,c}^{2i+ 1}(T_G^*M_{=i})=0$ for each $0\leq i\leq s-1$. Using this fact, both statements follow immediately from the long exact sequence of equivariant cohomology with compact supports associated to $j,e$.
\end{proof}

Let us now  make  a simple but important remark.
\begin{lemma}\label{laquetio} Let $\underline s\in \mathcal S_X(j)$ with $j<k$. The element $d_{X\setminus s}\in S[\mathfrak g^*]$ lies in the annihilator of
$H^*_{G,c}(T_G^*M_{X,\geq k})$.
\end{lemma}
\begin{proof} $H^*_{G,c}(T_G^*M_{X,\geq k})$ is a module over $H^*_{ G}(T_G^*M_{X,\geq k})$  and hence  also  over $H^*_{ G}( M_{X,\geq k})$. Thus this lemma follows from Lemma \ref{facil}.\end{proof}

Let us now split $X=A\cup B$ and $M_X=M_A\oplus M_B$. Then $$T^*M_X=M_X\times M_X^*=T^*M_A\times T^*M_B=T^*M_A\times M_B^*\times M_B.$$
Consider the inclusions
$$\begin{CD}T^*M_A @>s>>T^*M_A\times M_B^*@>i>>T^*M_A\times M_B^*\times M_B\end{CD}$$
Each is the zero section of a trivial bundle. The moment map for $T^*M_X$ restrict to $T^*M_A\times M_B^*$ to the moment map  of the factor $T^*M_A$. Hence setting $\tilde T^*_GM_A:=T^*_GM_A\times M_B^*$ we obtain the inclusions
$$\begin{CD}T_G^*M_A @>s>>\tilde T^*_GM_A @>i>>T_G^*M_X\end{CD}.$$
To the inclusion $i$ we can apply Proposition \ref{lamo} and to the inclusion $s$ also Remark \ref{ciccio}.

 In particular,  we get a Thom  isomorphism $$s_!=C_{M^*_B}: H_{G,c}^*(T_G ^* M_A)\to H_{G,c}^{*+ 2|B|}(\tilde T_G^*M_A)\cong  H_{G,c}^{*+ 2|B|}(T_G^*M_A\times  M_B^*) $$
  and  a homomorphism   $i_!:H_{G,c}^*(\tilde T^*_G  M_A)\to H_{G,c}^*(T^*_GM_X ).$
%Combining these 3 maps,   we claim that

\begin{proposition}\label{IMC} Take $\sigma\in H_{G,c}^*(T_G^*M_X )$,
then  $(i\circ s)_!C_{M^*_B}^{-1} i^*(\sigma)= d_B  \sigma$.

\end{proposition}
\begin{proof} We first  observe that $(i\circ  s)_!C_{M^*_B}^{-1}=  i_!  $.   So $(i\circ s)_!C_{M^*_B}^{-1} i^*(\sigma)=i_!i^*(\sigma)$ and the statement follows from Proposition \ref{lamo}.
   \end{proof}

\begin{corollary}\label{nablaindex}
Take
$\sigma\in H_{G,c}^*(T_G^*M_X )$.
Let $\sigma_0=C_{M^*_B}^{-1} i^*(\sigma)\in H_{G,c}^*(T^*_G M_A)$.
Then, we have the equality of distributions:
 $$ \partial_B  ( {\rm infdex}(\sigma) )={\rm infdex}(\sigma_0) .$$
\end{corollary}
\begin{proof} We use the fact that the infinitesimal index commutes with $i_!$ (see \cite{dpv34}, Theorem 4.9) and is a map of $S[\mathfrak g^*]$ modules.
\end{proof}

We  have defined in Definition \ref{gra} the space of distributions   $  {\mathcal G}(X)$ as those distributions $f$ on $\mathfrak g^*$  such  that $\partial_{X\setminus\underline r}f\in \mathcal S'(\mathfrak g^*, \underline r )$ for all $\underline t\in \mathcal S_X $ and $\tilde {\mathcal G}(X)$  as the $S[\mathfrak g^*]$ module generated by $  {\mathcal G}(X)$.
Then $\tilde {\mathcal G} _{i}(X)$ is the subspace in $\tilde {\mathcal G}(X)$ such that
$\partial_{X\setminus \underline t}f=0$ for all $\underline t\in \mathcal S_X{(i-1)}$.

\begin{lemma}\label{valindTotal} For each $  i\geq 0$,   ${\rm infdex}$ maps
$H_{G,c}^*(T^*_G M_{X,\geq i})$ to the space
$\tilde {\mathcal G}_{i}(X)$.
\end{lemma}
\begin{proof}
Denote  by $\ell:\tilde {\mathcal G} _{i}(X)\to \tilde {\mathcal G}(X)$ the inclusion.
By Lemma \ref{laquetio}, if $\sigma\in H_{G,c}^*(T_G^*M_{X,\geq i})$ and $\underline t$ is  a rational subspace of dimension strictly less than $i$,  we have $d_{X\setminus \underline t}  \sigma=0$. Thus  $\partial_{X\setminus \underline t}{\rm infdex}(\sigma)=0$.
It follows that the only thing we have to show is that, if $\sigma\in H_{G,c}^*(T^*_GM_X)$,  then
${\rm infdex}(\sigma)$  lies in $\tilde{\mathcal G}(X)$.
Take a rational subspace $\underline s$.  By Corollary \ref{nablaindex}, the infinitesimal index of $d_{X\setminus \underline s}  \sigma$ equals the infinitesimal index of an element $\sigma_0\in H_{G,c}^*(T^*_GM_{X\cap\underline s})$. But the action of $G$ on $M_{X\cap\underline s}$ factors though the quotient $G/G_{\underline s}$
whose Lie algebra is $\mathfrak g/\mathfrak g_{\underline s}$.
Thus $H_{G,c}^*(T_G^*M_{X\cap\underline s})\cong S[\mathfrak g^*]\otimes_{S[(\mathfrak g/\mathfrak g_{\underline s})^*]}H^*_{G/G_{\underline s},c}(T^*_{G/G_{\underline s}}M_{X\cap\underline s})$,
hence  $\partial_{X\setminus \underline s}{\rm infdex}(\sigma)\in S[\mathfrak g^*]
{\rm infdex}(H^*_{G/G_{\underline s},c}(T^*_{G/G_{\underline s}}M_{X\cap\underline s}))$.

But   ${\rm infdex}(H^*_{G/G_{\underline s},c}(T^*_{G/G_{\underline s}}M_{X\cap\underline s}))\subset  \mathcal S'(\mathfrak g^*, \underline r )$ hence the claim.
\end{proof}

The following   theorem characterizes the values of the infinitesimal index on  the entire $M_X$. This time,  we use the notations and the exact sequences contained in Theorem \ref{Tmenos} Remark \ref{lapes} and Corollary \ref{nablaindex}
\begin{theorem}\label{ilprinc} For each $0\leq i\leq s$,
\begin{itemize}
\item the diagram
$$\begin{CD}  0\to H_{G,c}^*(T_G^*M_{X,\geq i+1})@>{j_*}>> H_{G,c}^*(T_G^*M_{X,\geq i}) @>{C_{i}^{-1}e^*}>>H_{G,c}^*(T_G^* M_{=i})\to 0\\ @V{\rm infdex}VV@V{\rm infdex}VV@V{\rm infdex}VV@.\\
\hskip-1.0cm 0\to\tilde{\mathcal G}_{i+1}(X) @>\ell >> \tilde{\mathcal G}_{i}(X) @>  \mu_i>>
\oplus_{\underline s\in \mathcal S_X{(i)}}D^{\mathfrak g} (X\cap
\underline s )\to 0 \end{CD} $$
commutes.
\item Its vertical arrows are isomorphisms.

\item In particular, the infinitesimal index gives an isomorphism between  $H_{G,c}^*(T^*_GM_X)$ and $\tilde {\mathcal G}(X)$.

\end{itemize}

\end{theorem}
\begin{proof} Lemma \ref{valindTotal} tells us that the diagram is well defined. We need to prove commutativity.

We prove that  the square on the right hand side is commutative using   Corollary \ref{nablaindex}.
The square on the left hand side is commutative since $j_*$ is compatible with the infinitesimal index and $\ell$ is the inclusion.

Recall that $H_{G,c}^*(T_G^*M_{X\cap\underline s})\cong S[\mathfrak g^*]\otimes_{S[(\mathfrak g/\mathfrak g_{\underline s})^*]}H_{G/G_{\underline s},c}^*(T^*_{G/G_{\underline s}}M_{X\cap\underline s})$ and that $$D^{\mathfrak g}(X\cap
\underline s )=S[\mathfrak g^*] D (X\cap
\underline s )\cong S[\mathfrak g^*]\otimes_{S[(\mathfrak g/\mathfrak g_{\underline s})^*]}D(X\cap \underline s).$$  Using Theorem \ref{finitu},  this implies that the right vertical arrow is always an isomorphism.

We want to apply descending induction  on $i$. When $i+1=s$, since $M_{X,\geq s}= M_X^{fin}$ and $\tilde{\mathcal G}_{s-1}(X)=D(X)$, Theorem \ref{finitu}   gives that the left vertical arrow is an isomorphism. So assume that  the left vertical arrow is an isomorphism. We then deduce from the five Lemma that the central vertical arrow is an isomorphism and conclude.
\end{proof}

%\begin{remark}   Fourier transforms of elements in
% $\tilde{\mathcal G}_{i}(X)$  are supported on the closed set  $\cup_{\underline s\in\mathcal S^X_i}\underline s$.
%  This is again in agreement with the fixed point philosophy that we  recalled in Remark \ref{supportfourier}.  In fact,
%an element  $g\in \mathfrak g$ gives a 1--parameter group  with a fixed point   in $M_{\geq i}$ if and only if $g\in \cup_{\underline s\in\mathcal %S^X_i}$.
%\end{remark}


\begin{thebibliography}{99}

\bibitem{At} {Atiyah M.,}
{Elliptic operators and compact groups} , Springer L.N.M., n. 401, 1974.\newline
%
%
%\bibitem{Ats}  Atiyah, M. F.; Singer, I. M. The index of elliptic operators. I. Ann. of Math. (2) 87 1968 484--530.\newline
%
%
%
%\bibitem{Ats1} Atiyah, M. F.; Singer, I. M. Index theory for skew-adjoint Fredholm operators. Inst. Hautes tudes Sci. Publ. Math. No. 37 1969 5--26. \newline
%
%
%\bibitem{Ats2}  Atiyah, M. F.; Singer, I. M. The index of elliptic operators. III. Ann. of Math. (2) 87 1968 546--604. \newline
%
%\bibitem{BV2} Berline, Nicole; Vergne, Mich\`ele The Chern character of a transversally elliptic symbol and the equivariant index. Invent. Math. 124 (1996), no. 1-3, 11--49. \newline
%
%\bibitem{BV1}
%Berline, Nicole; Vergne, Mich\`ele L'indice \'equivariant des op\'erateurs transversalement elliptiques.   Invent. Math. 124 (1996), no. 1-3, 51--101.\newline
%
%
%



%\bibitem{BV} {Boysal A., Vergne M.,}
%{Paradan's wall crossing formula for partition functions and
%Khovanski-Pukhlikov differential operator.}(To appear)\newline
%


 \bibitem{DM1}{
Dahmen W., Micchelli C., }
 {On the solution of certain systems of partial difference
  equations and linear dependence of translates of box splines},
 {\it Trans. Amer. Math. Soc.},      {\bf 292},
  (1985),
  {1}, 305--320.\newline

\bibitem{DM3}{
Dahmen W., Micchelli C., }{The number of solutions to linear
Diophantine equations and  multivariate splines, }{\it Trans. Amer.
Math. Soc.} {\bf 308} (1988), no. 2, 509--532.\newline




\bibitem{danilov}{ Danilov V.I.,}{ The geometry of toric
  varieties},{\sl
  Uspekhi Mat. Nauk} {\bf 33} (1978), no. 2 (200), 85--134.
                  ({\sl Russian Math. Surveys} {\bf 33}  (1978), 97--154.)
                  \newline



%\bibitem{dp3}{De Concini C., Procesi C., }{Toric arrangements.}
%{\it  Transform. Groups} {\bf 10} (2005), no. 3-4, 387--422.\newline

\bibitem{dp1}{De Concini C., Procesi C., }
Topics in hyperplane arrangements, polytopes and box-splines.
Universitext. Springer, New York, 2011. xx+384 pp.\newline

\bibitem{dpv}{De Concini C., Procesi C., Vergne M.} {Vector partition function and generalized Dahmen-Micchelli spaces}
  {\sl Transformation Groups}, Volume 15, N. 4, (2010)  pp. 751-773\newline

\bibitem{dpv1}{De Concini C., Procesi C., Vergne M.} {Vector partition functions and index of transversally elliptic operators.}
  {\sl Transformation Groups},  Volume 15, N.  4, (2010) pp.775-811
\newline

\bibitem{dpv34}{De Concini C., Procesi C., Vergne M.} {The infinitesimal index} arXiv:1003.3525
\newline



\bibitem{dpv6}{De Concini C., Procesi C., Vergne M.}  Box splines and the equivariant index theorem,   arXiv:1012.1049 \newline


\bibitem{EG} {Edidin, Dan, Graham, William}{
Equivariant intersection theory. }
{\it Invent. Math.}{\bf  131} (1998), no. 3, 595--634.
\newline

%\bibitem{Er1} { Ehrhart E., }{\it
%Polyn\~ omes arithm\'etiques et m\'ethode des poly\`edres en combinatoire, }Birkh\"auser, Basel, 1977\newline
%


\bibitem{FS} {   Fink, A.,   Speyer,   D.}  {  $K$-classes of
matroids and equivariant localization.}  arXiv:1004.2403.
\newline




\bibitem{mat-qui} { V. Mathai, D. Quillen},
 Superconnections, Thom classes, and equivariant differential
   forms, {\it
   Topology}, {\bf 25}, (1986), p. 85-110.


%\bibitem{P} {Paradan, Paul--\'Emile }
%{Note sur les formules de saut de Guillemin-Kalkman. }{\it
%C. R. Math. Acad. Sci. Paris}{\bf 339} (2004), no. 7, 467--472.
%\newline
%
%\bibitem{P1} {Paradan, Paul--\'Emile }
%{Jump formulas in Hamiltonian Geometry}{\it arXiv:math/0411306}\newline
%
%\bibitem{par2} {Paradan, Paul--\'Emile }
%{ Localization of the Riemann-Roch character.}
%{\it J. Funct. Anal.} {\bf 187} (2001), 442--509.
%\newline



%\bibitem{She} {Shephard, G. C.}{ Combinatorial properties of associated zonotopes. }{\it Canad. J. Math.}{\bf  26} (1974), 302--321.
%\newline
%
%
%
%\bibitem{SV1}{ Szenes A., Vergne M., }{Residue formulae for
%vector partitions and Euler-Maclaurin sums}. Formal power series and
%algebraic combinatorics (Scottsdale, AZ, 2001). {\it Adv. in Appl. Math.} {\bf 30}
%(2003), no. 1-2, 295--342.
\end{thebibliography}
\end{document}